\documentclass{article}
\usepackage{amsmath}
\usepackage{amssymb}
\usepackage{amsthm}
\usepackage{amsfonts}
\usepackage{diagbox}
\newtheorem{theorem}{Theorem}[section]
\newtheorem{proposition}[theorem]{Proposition}
\newtheorem{lemma}[theorem]{Lemma}
\newtheorem{corollary}[theorem]{Corollary}

\newtheorem{definition}[theorem]{Definition}

\title{Isometries of virtual quadratic spaces}
\author{M\'at\'e L. Juh\'asz\\[2mm]
Alfr\'ed R\'enyi Institute of Mathematics\\
Hungarian Academy of Sciences\\
\texttt{juhasz.mate.lehel@renyi.mta.hu}}
\date{2016-12-29}

\let\bb=\mathbb

\DeclareMathOperator{\Ort}{Ort}
\DeclareMathOperator{\Iso}{Iso}
\DeclareMathOperator{\Rad}{Rad}

\begin{document}

\maketitle

\begin{abstract}
In this article, we introduce a new object, a virtual quadratic space, and its group of isometries. They are presented as natural generalizations of quadratic spaces and orthogonal groups. It is then shown that by replacing quadratic spaces with virtual quadratic spaces, we can unify certain enumerative properties of finite fields, without distinguishing between even and odd characteristics, such as the number of non-isomorphic non-degenerate quadratic forms, and the order of groups of isometries.
\end{abstract}

\section{Introduction}

Quadratic forms over finite fields have different properties depending on whether the characteristic of the field is $2$ or not. However, several general statements can be made without referencing the characteristic of the field. In particular, in even dimension, the number of non-degenerate quadratic forms is $2$ up to isomorphism, and the number of elements in its group of isometries is a polynomial in the number of elements of the field. However, when the dimension of a quadratic space is odd, every bilinear form is degenerate in characteristic $2$, and the number of isometries for a non-degenerate quadratic form is different if the characteristic is even.

Given a quadratic space $(V,Q)$ and a subspace $F\subseteq V$, the group of isometries that fix $F$ are in a bijection with the isometries of $(F^\perp,Q|_{F^\perp})$, provided that the associated bilinear forms on $V$ and $F$ are non-degenerate. In characteristic $2$, this is only possible if $\dim F^\perp$ is even. However, we may still consider pairs $(V,F^\perp)$ where we only assume that $\dim V$ is odd, and from \ref{Tmain}, this gives a natural generalization of quadratic spaces. Furthermore, the order of orthogonal groups is given by a polynomial that does not depend on the characteristic, as seen in \ref{Tortgroup}.

\section{Quadratic spaces}

Let us fix some field and denote it by ${\bb K}$.

\begin{definition}
A {\bf quadratic space} is a pair $(V,Q)$ consisting of a vector space and a quadratic form on it. It has an {\bf associated bilinear from} defined as $B(x,y)=Q(x+y)-Q(x)-Q(y)$. The {\bf radical} of a quadratic form $Q$ is the set of vectors $v$ such that $Q(v+u)=Q(u)$, and it is denoted by $\Rad Q$. The {\bf direct sum} of two quadratic spaces $(V,Q)$ and $(V',Q')$ is $(V\oplus V',Q\oplus Q')$ such that $(Q\oplus Q')(v\oplus v')=Q(v)+Q(v')$. An {\bf isometry} is a linear map $\varphi\colon V\to V$ such that $Q(\varphi(v))=Q(v)$ for all $v\in V$. The {\bf group of isometries} is denoted by $\Iso(V)$.
\end{definition}

\begin{definition}
Let us denote the function $u\to B(v,u)$ by $v^*$. A bilinear form is {\bf non-degenerate} or {\bf regular} if for every non-zero vector $v\in V$, the linear function $v^*$ is non-zero. A quadratic form is {\bf non-degenerate} if its associated bilinear form is non-degenerate.
\end{definition}

We need a few properties of hyperbolic spaces:

\begin{definition}
A {\bf hyperbolic space} of dimension $2$ is a pair $(\Sigma,B)$ is such that the bilinear form $B$ takes the form $B(u,v)=u_1v_2+u_2v_1$ in some basis. A hyperbolic space of dimension $2k$ is the orthogonal direct sum of $k$ hyperbolic spaces of dimension $2$ each.
\end{definition}

\begin{lemma}
\label{Thyperemb}
Assume $V$ is a vector space with a non-degenerate bilinear form $B$, and $N<V$ is a subspace with $B|_N\equiv 0$. Then there is a hyperbolic subspace $\Sigma=N\oplus\widetilde{N}$ of dimension $2\dim N$, such that $N^\perp\cap\Sigma=N$.
\end{lemma}

\begin{proof}
We will construct the spaces $\widetilde{N}$ and $\Sigma$ recursively. Choose a vector $u$ in $N$. Since $B$ is non-degenerate, there is a vector $v$ such that $B(u,v)=1$. Since $B|_N\equiv 0$, clearly $v\not\in N$. Furthermore, $\Sigma_0=\langle u,v\rangle$ is such that $B|_{\Sigma_0^\perp}$ is non-degenerate, hence we may apply the construction to $V':=\Sigma_0^\perp$ and $N':=N\cap V'$, unless $N'=\{0\}$, in which case we are done. The construction will give us $\widetilde{N}'$ and $\Sigma'$ of dimension $2\dim N'=2(\dim N-1)$, since $N<u^\perp$ and $u\in v^\perp$, and we may choose $\Sigma:=\Sigma_0\oplus\Sigma'$, $\widetilde{N}:=\langle v\rangle\oplus\widetilde{N}'$.
\end{proof}

The following lemma shows that in characteristic $2$, one can not construct a non-degenerate quadratic form in odd dimensions:

\begin{lemma}
\label{Thyper2}
Given a non-degenerate quadratic space $(V,Q)$ in characteristic $2$, its associated bilinear form gives $(V,B)$ a hyperbolic space structure.
\end{lemma}

\begin{proof}
It can be checked that any quadratic space decomposes as the direct sum of quadratic spaces of two kinds: those of the form $Ax^2$ and $A(x^2+xy+By^2)$. Since $B(u,u)=0$ for any $u$ in characteristic $2$, the first kind may not appear in the decomposition, while the second kind gives a hyperbolic space. For details, see \cite{kitaoka1993}.
\end{proof}

Since the dimension of a hyperbolic space is even, a non-degenerate quadratic space must have even dimension.

\section{Virtual quadratic spaces}

\begin{definition}
A {\bf virtual quadratic space} is a tuple $(V,Q,U)$ or $(V,U)$ for short, with $U$ a subspace of a vector space $V$ and $Q$ a non-degenerate quadratic form on $V$. Its {\bf dimension} is $\dim U$. An {\bf isometry} of the virtual quadratic space is an isometry of $(V,Q)$ that fixes $U^\perp$. The group of isometries is denoted by $\Iso(V,U)$.
\end{definition}

In general, we are only interested in the quadratic subspace $(U,Q)$, and we only use $V$ as an aid in theorems. As such, we must establish whether such a $(V,Q)$ exists for a quadratic space $(U,Q)$, and whether it is unique. First let us look at the question of existence.

\begin{proposition}
\label{Tembed}
Any quadratic space $U$ can be embedded into some virtual quadratic space $(V,U)$.
\end{proposition}

\begin{proof}
Let us denote $N:=U\cap U^\perp$. Since $B|_N\equiv 0$, we shall embed $N$ into a hyperbolic space. Define $V:=\widetilde{N}\oplus U$ where $\widetilde{N}$ is isomorphic as a vector space to $N$. We shall give $\widetilde{N}\oplus N$ a hyperbolic space structure in the following way. Let us choose a basis $f_i$ for $N^*$ and the equivalent $\tilde f_i$ for $\widetilde{N}^*$, and extend the $f_i$ to $U$ in an arbitrary manner.
Now for $\tilde u\oplus v\in V$, with $u\in N$ and $v\in U$, define $\tilde Q(\tilde u\oplus v)=Q(v)+\sum \tilde f_i(\tilde u)f_i(v)$.
It can be verified that this quadratic form has non-degenerate associated bilinear form.
\end{proof}

Uniqueness of $V$ can of course not be guaranteed, but we may look for a {\it minimal} virtual quadratic space, and show some type of uniqueness for that. The following theorem shows what such a space looks like, and how we may characterize this minimality.

\begin{proposition}
\label{Tminimal}
For any virtual quadratic space $(V,Q,U)$ with associated bilinear form $B$, there is a subspace $U\le V_m\le V$ with $B|_{V_m}$ non-degenerate, such that $U^\perp\cap V_m\subseteq U$. For such a subspace, $\Iso(V,U)=\Iso(V_m,U)$. A virtual quadratic space $(V,U)$ is called {\bf minimal} if $U^\perp\subseteq U$, and $\dim V=\dim U+\dim(U\cap U^\perp)$ only depends on $Q|_{U}$ for minimal virtual quadratic spaces.
\end{proposition}

\begin{proof}
We don't actually need $Q$ in the proof of the first statement, and may only look at $B$. We will proceed by constructing the same space as in \ref{Tembed}. Let us fix $N:=U\cap U^\perp$. Since $B|_N\equiv 0$, by \ref{Thyperemb} we may embed $N$ into a subspace $\Sigma=N\oplus\widetilde{N}$ of dimension $2\dim N$, and $B|_\Sigma$ is clearly non-degenerate, giving $\Sigma\cap\Sigma^\perp=\{0\}$.

We can see that $U\cap\Sigma=N$, since $U\cap\Sigma\supset N$ by the definition of $N$ and $\Sigma$, and if $u\in U\cap\Sigma$, then $u\perp N$, but $u\in N^\perp\cap\Sigma=N$ by the construction of $\Sigma$. For similar reasons, $U^\perp\cap\Sigma=N$.

Therefore $U$ decomposes as orthogonal subspaces $N\oplus M$ with $M=U\cap\Sigma^\perp$. Furthermore, $M\cap M^\perp=\{0\}$, since for all $u\in M\subseteq U\setminus N$, there is a $v\in U$ such that $u\not\perp v$. Then $v$ decomposes as $v_M\oplus v_N$ with $v_M\in M$ and $v_N\in N$, and since $u\perp v_N$, we have $u\not\perp v_M\in M$. Therefore $B|_M$ is non-degenerate. Also, $V$ decomposes as orthogonal subspaces $M\oplus \widehat{M}\oplus\Sigma$ with $\widehat{M}=M^\perp\cap\Sigma^\perp$.

Let us define $V_m:=U+\Sigma$, which is $M\oplus\Sigma$ as an orthogonal decomposition. Clearly $(V_m, U)$ is a virtual quadratic space, since $(V_m,Q)$ is non-degenerate as $V_m\cap V_m^\perp=\{0\}$. Since $U^\perp\cap\Sigma=N\subseteq U$, we have $U^\perp\cap(M\oplus\Sigma)\subseteq U$, and so $(V_m,U)$ is minimal.

Clearly $\dim V_m=\dim M+\dim\Sigma=\dim U+\dim N$. Given a virtual quadratic space $(V, U)$, the above construction gives a minimal subspace $(V_m,U)$. Since $V$ decomposes as orthogonal subspaces $M\oplus \widehat{M}\oplus\Sigma$, we get $U^\perp=M^\perp\cap N^\perp=(\widehat{M}\oplus\Sigma)\cap(M\oplus N\oplus\widehat{M})=\widehat{M}\oplus N$. Then $U\subseteq U^\perp$ if and only if $\dim\widehat{M}=0$ and $V=V_m$. Therefore the dimension formula holds for all minimal spaces.

Finally, since $U^\perp=\widehat{M}\oplus N$, $\Iso(V,U)$ fixes $U^\perp\supseteq\widehat{M}$, and the action restricts to $\widehat{M}^\perp$, giving $\Iso(V,U)=\Iso(\widehat{M}^\perp,U\cap \widehat{M}^\perp)=\Iso(V_m,U)$.
\end{proof}

This proof tells us not only that a minimal virtual quadratic space exists, it also shows us the structure it has, and we shall use the notations introduced in this proof in other propositions as well.

However, since the theorem extends only the bilinear form to $V$, and in characteristic $2$ that does not define the quadratic form, the minimal virtual quadratic space $(V,U)$ containing $U$ is still not unique. Fortunately, this is not an issue, at least for the isometry groups, as seen from the following theorem.

\begin{proposition}
\label{Tisounique}
For any non-degenerate virtual quadratic space $(V,U)$, the group $\Iso(V,U)$ depends only on $Q|_{U}$. In particular, if given two non-degenerate virtual quadratic spaces $(V,Q,U)$ and $(V',Q',U')$, with $(U,Q|_U)$ and $(U',Q'|_{U'})$ isomorphic as quadratic spaces, then there is a bijection between $\Iso(V,U)$ and $\Iso(V',U')$.
\end{proposition}

\begin{proof}
Consider two virtual quadratic spaces $(V,U)$ and $(V',U')$ such that $Q|_{U}\cong Q'|_{U'}$. First we may assume that they are both minimal, as all such spaces have a minimal subspace with an isometry group isomorphic to the isometry group of the containing space, in which case $\dim V=\dim V'$. We may assume that $V=V'$ and $U=U'$. Then by introducing the difference $\delta Q=Q'-Q$, we get $\delta Q|_{U}\equiv 0$.

If the characteristic is not $2$, then two quadratic spaces are isomorphic if an isomorphism of the vector spaces identifies their bilinear forms. By using the notations of the previous proof, $U$ decomposes as a direct sum of orthogonal subspaces $N\oplus M$, where $N=U\cap U^\perp$ and $\Sigma$ is a hyperbolic subspace of dimension $2\dim N$ containing $N$.
Since the pair $(V,U)$ is minimal, $V=M\oplus\Sigma$. By introducing the analoguous symbols for the pair $(V',U')$, all the components are isomorphic, hence we can choose the isomorphism that sends $V=M\oplus\Sigma$ to $V'=M'\oplus\Sigma'$ component-wise. Under this isomorphism, $Q|_V=Q'|_V$, so the theorem becomes a trivial condition.

Let us look at the characteristic $2$ case, where the associated bilinear form is always hyperbolic by \ref{Thyper2}.
Since $B|_{U}=B'|_{U'}$, we may fix an identification between $V$ and $V'$ that extends the isometric map $U\to U'$ in a way such that $B=B'$. Then the map $\delta Q(x):=Q'(x)-Q(x)$ is an additive map.

Let us choose an automorphism $\varphi\in\Iso_Q(V,U)$. In order to prove that $\varphi$ is also in $\Iso_{Q'}(V,U)$, it is sufficient to show that $\varphi$ preserves the function $\delta Q$, since then $Q'(\varphi(x))=Q(\varphi(x))+\delta Q(\varphi(x))=Q(x)+\delta Q(x)=Q'(x)$.

Since $U^\perp$ is fixed under the automorphism $\varphi$, the scalar product functions $B(u,.)$ are preserved for all $u\in U^\perp$, and in particular, the subspace $U$ is preserved. Equivalently, the automorphism acts trivially on the quotient vector space $V/U$, since for any $x$, assuming that $\delta x:=\varphi(x)-x\not\in U$, there is at least one associated $u\in U^\perp$ such that $B(u,\delta x)\ne 0$, which would contradict the preservation of the function $B(u,.)$.

Since the function $\delta Q$ is additive and vanishes on $U$, it is well defined on the quotient additive group $V/U$, which is naturally isomorphic to the quotient vector space $V/U$. On $V/U$, $\varphi$ acts trivially, and so $\delta Q$ is preserved.
\end{proof}

These two propositions motivate the following definition:

\begin{definition}
Two virtual quadratic spaces $(V,Q,U)$ and $(V',Q',U')$ are {\bf isomorphic} if $(U,Q|_U)$ and $(U',Q|_{U'})$ are isomorphic quadratic spaces.
\end{definition}

Now we will show a relationship between the isometry groups $\Iso(U)$ and $\Iso(V,U)$. It is known (see \cite{kitaoka1993}) that in a non-degenerate quadratic space, the group of isometries acts transitively on non-zero vectors of equal norm:

\begin{theorem}
\label{Tkitaoka}
Given a quadratic space $U$ and two subspaces $V$ and $W$ such that $V\cap U^\perp=W\cap U^\perp=\{0\}$ with an isometry $\sigma\colon V\to W$, it extends to an isometry of $U$.
\end{theorem}

\begin{proposition}
\label{Tisosurj}
Given a quadratic space $(U,Q)$, and assuming that for any non-zero $u\in U\cap U^\perp$ we have $Q(u)\ne 0$, then every isometry of $U$ fixes $U\cap U^\perp$. Furthermore, if given an embedding of $U$ into a minimal virtual quadratic space $(V,U)$, there is a natural map $\Iso(V,U)\to\Iso(U)$ that is surjective.
\end{proposition}

\begin{proof}
We will first show that $Q$ as a map is an injection from $N:=U\cap U^\perp$ to the base field. In fact, given $u$, $v\in N$, $Q(u-v)=Q(u)-Q(v)$ since $u\perp v$. Therefore if $Q(u)=Q(v)$, we get $Q(u-v)=0$, which is only possible if $u=v$. Since $Q$ is injective, and any isometry maps $N$ to itself, it must fix each vector, thus fixing $N$.

Given a virtual quadratic space $(V,U)$, any isometry of $V$ fixes $N\subseteq U^\perp$. If $(V,U)$ is minimal, $N^\perp=U$, hence the subspace $U$ is preserved, and there is a restriction map $\Iso(V,U)\to\Iso(U)$. Since $(Q,V)$ is a non-degenerate quadratic space, by \ref{Tkitaoka}, any isometry of $U$ extends to an isometry of $V$. Hence the restriction map is surjective.
\end{proof}

\begin{proposition}
\label{Tradical}
The condition that for any non-zero $u\in U\cap U^\perp$, $Q(u)\ne 0$, is equivalent to $\Rad Q=\{0\}$. If the characteristic is not $2$, this is equivalent to $U\cap U^\perp=\{0\}$.
\end{proposition}

\begin{proof}
In fact we shall prove that $\Rad Q=\{u\in U\cap U^\perp\mid Q(u)=0\}$. In one direction, if $u\in\Rad Q$, then $B(u,v)=Q(u+v)-Q(u)-Q(v)=0$ and $Q(u)=Q(u+0)=Q(0)=0$. Now assume $Q(u)=0$ and $B(u,v)=0$ for all $v\in U$. Then $Q(u+v)=Q(u)+Q(v)+B(u,v)=Q(v)$, hence $u\in\Rad Q$.

If the characteristic is not $2$, the bilinear form $B$ defines $Q$, and so $Q|_{U\cap U^\perp}\equiv 0$ since $B|_{U\cap U^\perp}\equiv 0$.
\end{proof}

Propositions \ref{Tembed}, \ref{Tminimal}, \ref{Tisounique}, \ref{Tisosurj} and \ref{Tradical} may be combined into the following theorem:

\begin{theorem}\label{Tmain}
Consider a quadratic space $(Q_U,U)$. Then it can be embedded into a virtual quadratic space $(V,Q_V,U)$, and for such an embedding, $\Iso(V,U)$ depends only on $Q_U$. In fact, such a $V$ can always be chosen so that $\dim V=\dim U+\dim(U\cap U^\perp)$, even as a subspace of some other virtual quadratic space $(V',Q_{V'},U)$, which is equivalent to the condition that $U^\perp\subseteq U$ in $V$. Furthermore, if $\Rad Q_U=\{0\}$, the restriction map $\Iso(V,U)\to\Iso(U)$ is a surjective map.
\end{theorem}

Proposition \ref{Tisosurj} motivates also the following definition.

\begin{definition}
A virtual quadratic space $(V,U)$ is {\bf non-degenerate} if $\Rad Q=\{0\}$.
\end{definition}

\section{Finite fields}

One interesting application of virtual quadratic spaces is that they provide a common language for finite fields of even and odd characteristic. First, consider the following lemma.

\begin{lemma}
\label{Tvirtual2}
In a prefect field ${\bb K}$ of characteristic $2$, any non-degenerate virtual quadratic space $(V,U)$ is such that $\dim(U\cap U^\perp)\le 1$.
\end{lemma}

\begin{proof}
Assume that there are two linearly independent vectors $u$, $v\in U\cap U^\perp$ with $Q(u)\ne 0\ne Q(v)$. Since ${\bb K}$ is perfect, there is an element $\lambda\in{\bb K}$ such that $\lambda^2={Q(u)\over Q(v)}$. Then the vector $w:=u+\lambda v$ has $Q(w)=0$. Since $(V,U)$ is non-degenerate, this contradicts the fact that $u$ and $v$ are linearly independent.
\end{proof}

Now let us recall a few simple theorems. See \cite{kitaoka1993} for details.

\begin{theorem}
\label{Tquadform}
Let us fix a finite field ${\bb F}$ of odd characteristic, and choose a non-square element ${\bf e}\in{\bb F}$. Then every non-degenerate quadratic form is of one of the following forms, up to isomorphism:\hfill\break
$\bullet$ $\sum_{i=1}^k x_{2i-1}x_{2i}$ for $n=2k$;\hfill\break
$\bullet$ $\sum_{i=1}^k x_{2i-1}x_{2i}+x_{2k+1}^2$ for $n=2k+1$;\hfill\break
$\bullet$ $\sum_{i=1}^k x_{2i-1}x_{2i}+x_{2k+1}^2-{\bf e}x_{2k+2}^2$ for $n=2k+2$.
\end{theorem}

\begin{theorem}
\label{Tquadform2}
Let us fix a finite field ${\bb F}$ of characteristic $2$, and choose an element ${\bf e}\in{\bb F}$ for which the polynomial $x^2+x+{\bf e}$ has no roots. Then every quadratic form with trivial radical is of one of the following forms, up to isomorphism:\hfill\break
$\bullet$ $\sum_{i=1}^k x_{2i-1}x_{2i}$ for $n=2k$;\hfill\break
$\bullet$ $\sum_{i=1}^k x_{2i-1}x_{2i}+x_{2k+1}^2$ for $n=2k+1$;\hfill\break
$\bullet$ $\sum_{i=1}^k x_{2i-1}x_{2i}+x_{2k+1}^2+x_{2k+1}x_{2k+2}+{\bf e}x_{2k+2}^2$ for $n=2k+2$.\hfill\break
The first and last have non-degenerate associated bilinear forms, but the second has not.
\end{theorem}

The first and last cases are denoted for all fields as $+$-type and $-$-type, respectively. These two theorems can be combined into the following corollary:

\begin{corollary}
Let us fix a finite field ${\bb F}$. Then the number of non-degenerate virtual quadratic spaces of dimension $n$ up to isomorphisms is $2$ if $n$ is even and $1$ if $n$ is odd.
\end{corollary}

\begin{proof}
We may assume that the virtual quadratic space is minimal. A non-degenerate virtual quadratic space is a triple $(V,Q,U)$ such that $Q(u)=0$ for $u\in U\cap U^\perp$ only if $u=0$. If the characteristic of the field is odd, this is only possible if $V=U$, hence this is the same case as theorem \ref{Tquadform}. If the characteristic is $2$, $V$ must be of even dimension by theorem \ref{Tquadform2}. Since all finite fields are perfect, by \ref{Tvirtual2} we have $\dim V-\dim U\le 1$ since $V$ is minimal. Hence if $\dim U$ is even, then $V=U$.

Assume that $\dim U$ is odd and $\dim V=\dim U+1$. Then $\Rad Q|_U=\{0\}$, since the virtual quadratic space is non-degenerate. Hence $Q|_U$ is of the form prescribed in \ref{Tquadform2}, which determines the virtual quadratic space up to isomorphism.
\end{proof}

The order of orthogonal groups over finite fields is known (see \cite{huppert1967}).

\begin{theorem}
Consider a quadratic space $(U,Q)$ with $\Rad Q=\{0\}$ over the finite field ${\bb F}_q$, and let us denote by ${\rm O}^\varepsilon(2k,q)$ the group $\Iso(U)$ when $\dim U=2k$ and $Q$ is of type $\varepsilon$, and by ${\rm O}(2k+1,q)$ the group $\Iso(U)$.
Then
$${\rm O}^\varepsilon(2k,q)=2q^{k^2-k}(q^{k}-\varepsilon)\prod_{i=1}^{k-1}(q^{2i}-1)$$
If $2\nmid q$,
$${\rm O}(2k+1,q)=2q^{k^2}\prod_{i=1}^{k}(q^{2i}-1)$$
If $2\mid q$,
$${\rm O}(2k+1,q)=q^{k^2}\prod_{i=1}^{k}(q^{2i}-1)$$
\end{theorem}

The formula for even dimension does not discriminate between even and odd characteristics, but the formula for odd dimension does. Virtual quadratic spaces give us a hint that the problem is that the space has degenerate associated bilinear form, and as it turns out, it is:

\begin{theorem}
Let $(V,U)$ be a non-degenerate virtual quadratic space of dimension $2k+1$ over a field ${\bb F}_q$ of characteristic $2$. Then
$$|\Iso(V,U)|=2q^{k^2}\prod_{i=1}^{k}(q^{2i}-1)$$
\end{theorem}

\begin{proof}
It is known from \ref{Tmain} that the restriction map $\Iso(V,U)\to\Iso(U)$ is a surjection, and that $\Iso(U)\cong{\rm O}(2k+1,q)$, hence we only need to show that the kernel of the restriction map is of order $2$.

Assume that $(V,U)$ is minimal, and let us take an isometry $\varphi\in\Iso(V,U)$ that is in the kernel, hence it fixes $U$. We may decompose $V$ as the orthogonal sum $M\oplus\Sigma$ where $\Sigma$ is a hyperbolic space, and $U$ as $M\oplus N$ where $N=U\cap U^\perp$. Then $\varphi$ fixes $M\subset U$, and thus preserves the subspace $\Sigma$.

By \ref{Tvirtual2}, $\dim N=1$, and $\Sigma$ has a basis $\{e_1,e_2\}$ with $\langle e_1\rangle=N$, where the bilinear form takes the form $B(u,v)=u_1v_2+u_2v_1$. Since $(V,U)$ is non-degenerate, $\Rad Q|_U=\{0\}$, and $Q(e_1)\ne 0$, in fact we may assume $Q(e_1)=1$ by rescaling, as the field is perfect. Since $\varphi$ fixes $U^\perp$, which contains $e_1$, we only need to check the image of $e_2$. Let $\varphi(e_2)=\alpha e_1+\beta e_2$ for some parameters $\alpha$, $\beta\in{\bb F}_q$.

First of all, $1=B(e_1,e_2)=B(e_1,\varphi(e_2))=\beta$. Then $Q(e_2)=Q(\varphi(e_2))=\alpha^2+\alpha\beta+\beta^2Q(e_2)$, which gives us $\alpha(\alpha+1)=0$, hence $\alpha\in\{0,1\}$. Since either choice gives us an isometry, we have the kernel containing $2$ elements.
\end{proof}

\begin{corollary}\label{Tortgroup}
Consider a non-degenerate virtual quadratic space $(V,Q,U)$ the finite field ${\bb F}_q$, and let us denote by $\Ort^\varepsilon(2k,q)$ the group $\Iso(V,U)$ when $\dim U=2k$ and $Q$ is of type $\varepsilon$, and by $\Ort(2k+1,q)$ the group $\Iso(V,U)$.
Then
$$\Ort^\varepsilon(2k,q)=2q^{k^2-k}(q^{k}-\varepsilon)\prod_{i=1}^{k-1}(q^{2i}-1)$$
$$\Ort(2k+1,q)=2q^{k^2}\prod_{i=1}^{k}(q^{2i}-1)$$
\end{corollary}

\section*{Acknowledgement}

I am grateful for the help of J\'ozsef Pelik\'an for providing me with references.

\end{document}